\documentclass{amsart}

\usepackage{tikz}
\usepackage{xcolor}

\usepackage{latexsym,amsmath,amssymb,epic}
\usepackage{latexsym,amssymb,epic}
\usepackage{amsmath,amsthm}
\usepackage{color}
\usepackage{soul}
\usepackage[pagebackref]{hyperref}
\hypersetup{
	colorlinks=true, %set true if you want colored links
	linktoc=all, %set to all if you want both sections and subsections linked
	linkcolor=blue} %choose some color if you want links to stand out

\usepackage[margin=1in]{geometry}

\usepackage[numbers, sort&compress]{natbib}

\newtheorem{theorem}{Theorem}[section]
\newtheorem{lemma}[theorem]{Lemma}
\newtheorem{corollary}[theorem]{Corollary}

\newcommand{\OL}[1]{\overline{#1}}
\newcommand{\Z}{{\mathbb Z}}

\newcommand{\infprod}{{\displaystyle \prod_{i=1}^{\infty}}}

\allowdisplaybreaks

\begin{document}

	\title[Congruences for Overcubic Partition $k$-Tuples]{Congruences for Overcubic Partition $k$-Tuples}

	\author{Daniel Chac\'{o}n}
	\address{Department of Mathematics and Statistics, University of Minnesota Duluth, Duluth, MN 55812, USA}
	\email{chaco054@d.umn.edu }

	\author{James A. Sellers}
	\address{Department of Mathematics and Statistics, University of Minnesota Duluth, Duluth, MN 55812, USA}
	\email{jsellers@d.umn.edu}

	

\begin{abstract}
\noindent In the last few years, a number of authors have proved divisibility properties satisfied by various functions which count the number of overcubic partition $k$--tuples of weight $n$ for small values of $k$.  In this work, we use generating functions to prove some of their results as well as multiple infinite families of new congruences for   overcubic partition $k$-tuples which do not yet appear in the literature.  In particular, we focus on a new perspective which provides insights as to why these functions are often divisible by powers of 2, and we also prove families of congruences whose moduli are odd.  For example, we prove that, for all $m\geq 0$, $\OL{b}_{4}(22m + 11) \equiv 0 \pmod{11}$ and we also prove infinite families such as $\OL{b}_{9l+2}(9m+3) \equiv 0 \pmod{3}$ for all $m,l\geq0$.

\end{abstract}

\maketitle

\section{Introduction} 
\label{sec:Introduction}

A partition $\lambda$ of a positive integer $n$ is a finite non-increasing sequence of positive integers $\lambda_1 \geq \lambda_2 \geq \cdots \geq \lambda_r$ such that $\sum\limits_{i=1}^r \lambda_i = n$. Each $\lambda_i$ is called a part of the partition $\lambda$.
For example, there are five  partitions when $n = 4$ which are
$$4 \text{, } 3+1\text{, }2+2\text{, }2+1+1, \text{ and } 1+1+1+1.$$
For all $n\geq0$, let $p(n)$ be the number of partitions of $n$ with $p(0) := 1$.  From the example above, we have $p(4) = 5$.

In the work below, we will be interested in the generating functions for various integer partition functions.  As such, we will utilize $q$-Pochhammer notation which is defined as follows:  For $j \geq 1$,
$$(a;q)_j = \prod_{i=1}^j (1-aq^{i-1})$$
where $(a;q)_0 = 1$.  For $|q|<1$, we have  
$(a;q)_\infty = \lim_{j\to\infty}(a;q)_j.$  To simplify our notation even further, we define $f_k = (q^k;q^k)_\infty$ for fixed $k\geq 1$.  

As Euler proved, the generating function for $p(n)$ is given by
$$\sum_{n\geq0}p(n)q^n = \frac{1}{f_1}.$$
Approximately a century ago, Ramanujan \cite{Ramanujan} noticed congruences satisfied by $p(n)$ for the moduli $5, 7,$ and $11$.  Using generating function manipulations, he proved the following:  

\begin{theorem}\label{ramanujanCongruences}
    For all $n\geq 0$, 
    \begin{align}
        p(5n+4) &\equiv 0 \pmod{5}, \\
        p(7n+5) &\equiv 0 \pmod{7}, \text { and}\\
        p(11n+6) &\equiv 0 \pmod{11}.
    \end{align}
\end{theorem}
These congruences, and others that have been proven for various partition functions since the time of Ramanujan, serve as the motivation for the work below.  

In order to describe the family of functions that we wish to study, we must provide a few more definitions.  

An overpartition of $n$ is an integer partition of $n$ where the first occurrence of any part may be overlined.  To illustrate, there are eight overpartitions of 3:
$$3\text{, } \OL{3}\text{, } 2+1\text{, }\OL{2}+1\text{, }2+\OL{1}\text{, }\OL{2}+\OL{1}\text{, }1+1+1, \text{ and } \OL{1}+1+1.$$
We define $\overline{p}(n)$ to be the number of overpartitions of $n$, so we have $\overline{p}(3)=8$ thanks to the example above.  
%%% ****************************

Overpartitions and their congruences have been well-studied; 
see \cite{Chen2015, Chen2016, CorteelLovejoy, Dasappa2024, Dou2017, overpartitionArith, overpartitionMod12, overpartitionMod5and9, KimOverpartMod128, ByungchanKimOverpartPairsModTwo,  Liang2025, Lin2015, Shomanov2025, Xia2017, Yang2017, Yao2018,  Zhang2025} 
for examples of such work. The generating function for $\overline{p}(n)$ is given by 
    $$\sum_{n\geq0}\OL{p}(n)q^n = \frac{(-q;q)_\infty}{(q;q)_\infty}=\frac{f_2}{f_1^2}.$$

In 2010, Hei-Chi Chan \cite{HeiChiChanBeautiful, HeiChiChanCubic} introduced the cubic partition function $b(n)$ in connection with Ramanujan's cubic continued fraction.  The generating function for $b(n)$ is given by 
\begin{equation}
\label{gen_fn_A}
B(q) := \sum_{n\geq 0} b(n) q^n = \frac{1}{(q;q)_\infty(q^2;q^2)_\infty} = \frac{1}{(q;q^2)_\infty(q^2;q^2)_\infty^2}.
\end{equation}
It is clear that $b(n)$ also counts the number of integer partitions of $n$ where each even part may appear in one of two colors. For example, $b(4) = 9$ thanks to the following:  
$$4_1\text{,\ } 4_2\text{,\ } 3+1\text{,\ } 2_1+2_1\text{,\ }  2_2+2_1\text{,\ } 2_2+2_2\text{,\ } 2_1+1+1\text{,\ } 2_2+1+1, \text{ and } 1+1+1+1$$
Note that the subscripts represent the two possible colors for the even parts.  
Many congruences and related arithmetic properties satisfied by $b(n)$ have been proven  \cite{HeiChiChanBeautiful, HeiChiChanCubic, ChiranjitRay, JamesSellers, ZhaoCubic}.

Soon after the work of Chan appeared, Kim \cite{KimOvercubicMod3} combined the ideas of overpartitions and cubic partitions to define overcubic partitions.  As one might imagine, these are partitions where the first occurrence of each distinct colored part type may be overlined and each even part comes in one of two colors.
For example, there are twelve overcubic partitions of 3 which are
$$3\text{, } \OL{3}\text{, } 2_1+1\text{, }\OL{2}_1+1,2_1+\OL{1}\text{, }\OL{2}_1+\OL{1}\text{, } 2_2+1\text{, }\OL{2}_2+1,2_2+\OL{1}\text{, }\OL{2}_2+\OL{1}\text{, }1+1+1, \text{ and } \OL{1}+1+1.$$

Thus, if we define $\OL{b}(n)$ to be the number of overcubic partitions, then the above example yields $\OL{b}(3) = 12$.  The generating function for $\OL{b}(n)$ is given by 
\begin{equation}
\label{overcubic_gen_fn}
 \OL{B}(q)=\sum_{n\geq0}\OL{b}(n)q^n = \frac{f_4}{f^{2}_1f_2}.   
\end{equation}
As is the case for its counterparts $\OL{p}(n)$ and $b(n)$, the function $\OL{b}(n)$ satisfies many arithmetic properties; see, for example, \cite{HirschhornNoteOvercubic, KimOvercubicMod3, JamesSellers, CShivashankar}.

Prior to generalizing to $k$--tuples of partitions (which is the ultimate goal of this work), we pause here and highlight specific work that has been completed in the past related to $b(n)$ and $\OL{b}(n)$ which will prove useful below.  In order to describe these results, we must remind the reader of two of Ramanujan's theta functions, $\varphi(q)$ and $\psi(q)$, which are defined as follows:
\begin{equation}
\label{phi_sum_defn}
    \varphi(q) = 1 + 2\sum_{n\geq1}q^{n^2},
\end{equation}
\begin{equation}
\label{psi_sum_defn}
\psi(q) = \sum_{k\geq0} q^{\frac{k^2+k}{2}}.
\end{equation}
These two functions play a significant role in the work below.  Using these two theta functions, Sellers proved the following functional equation \cite[Theorem 2.1]{JamesSellers}.
\begin{theorem}
    We have $B(q) = \psi(q)\psi(q^2)B(q^2)^2$.
\end{theorem}
Through repeated application of this theorem, he proved the corollary below \cite[Corollary 2.2]{JamesSellers}.
\begin{corollary}
    We have $B(q) = \psi(q) \prod_{i\geq1}\psi(q^{2^i})^{3\cdot2^{i-1}}$.  
\end{corollary}
Sellers also proved a functional equation for $\OL{B}(q)$ similar to the one for $B(q)$ \cite[Theorem 2.1]{JamesSellers}.
\begin{theorem}
\label{b_1 phi recurrence}
    We have $\OL{B}(q) = \varphi(q)\varphi(q^2)\OL{B}(q^2)^2$.
\end{theorem}
Through repeated application of this theorem, he proved this corollary  \cite[Corollary 2.3]{JamesSellers}.
\begin{corollary}
\label{b_1 phi defn}
    We have $\OL{B}(q) = \varphi(q) \prod_{i\geq1}\varphi(q^{2^i})^{3\cdot2^{i-1}}$.
\end{corollary}
Theorem \ref{b_1 phi recurrence} and Corollary \ref{b_1 phi defn} will appear again in the discussion below.  
In addition, Sellers proved the following characterizations modulo 2 and 4 \cite [Theorem 2.4--2.5]{JamesSellers}.
\begin{theorem}\label{b_1 mod 2 characterization}
    For all $n\geq1$, $\OL{b}(n)\equiv0\pmod{2}$.
\end{theorem}
\begin{theorem}\label{b_1 mod 4 characterization}
    For all $n\geq1$,
    \[ \OL{b}(n) \equiv \begin{cases}
                    2 \pmod{4} \text{ if } n = l^2 \text{ or } n = 2l^2 \text{, } l\in\Z \text{ and} \\
                    0 \pmod{4} \text{ otherwise}.
                    \end{cases} \]
\end{theorem}
%\begin{theorem}\label{b_1 mod 8 characterization}
%    For all $n\geq 1$,\[ \OL{b}(n) \equiv \begin{cases}
%                    2 \pmod{8} \text{ if } n = l^2 \text{ or } n = 2(2l)^2 \text{, } l\in\Z  \text{ with } l\geq1, \\
 %                   6 \pmod{8} \text{ if } n = 2(2l-1)^2 \text{, } l\in\Z  \text{ with } l\geq1, \\
%                    4 \pmod{8} \text{ if } n = l^2+2j^2 \text{, } l,j\in\Z  \text{ with } l,j\geq1\text{, and} \\
%                    0 \pmod{4} \text{ otherwise}.
%                    \end{cases} \]
%\end{theorem}
\noindent 
Theorem \ref{b_1 phi recurrence}, Corollary \ref{b_1 phi defn}, Theorem \ref{b_1 mod 2 characterization}, and Theorem \ref{b_1 mod 4 characterization} will be fully generalized when we transition to overcubic partition $k$--tuples, the primary set of combinatorial objects under consideration in this work.  

A partition $k$-tuple of weight $n$ is an ordered $k$-tuple of partitions whose total sum is $n$.  Note that some of the $k$ partitions in such a $k$-tuple may be empty; in that case, we denote such an empty partition with the symbol $\emptyset$.  To illustrate, there are sixteen overcubic partition doubles (or 2--tuples) of weight 2 which are given as follows:
\begin{align*}
    &(2_1,\emptyset)\text{, } (\OL{2}_1,\emptyset)\text{, } (2_2,\emptyset)\text{, } (\OL{2}_2,\emptyset)\text{, } (\emptyset,2_1)\text{, }  (\emptyset,\OL{2}_1)\text{, } (\emptyset,2_2)\text{, } (\emptyset,\OL{2}_2), \\
    &(1+1,\emptyset)\text{, }  (\OL{1}+1,\emptyset)\text{, } (1,1)\text{, } (\OL{1},1)\text{, } (1,\OL{1})\text{, } (\OL{1},\OL{1})\text{, }  (\emptyset,1+1),\text{ and } (\emptyset,\OL{1}+1).
\end{align*}

% Note that partition tuple functions can be made for any of the previous partitions and have been extensively studied \cite{ByungchanKimOverpartPairsModTwo, DasOverpartOdd, ChiranjitRay, LinCubicPairsMod27, ChernCubicPairsMod3, LinCubicPairsMod3, ChenOverPartPairs}.  

Based on the generating function for $b(n)$ given in Equation \eqref{overcubic_gen_fn}, it should be clear that the generating function for $\OL{b}_k(n)$, the number of overcubic partition $k$-tuples of weight $n$, is given by 
\begin{equation}
\label{overcubic-k-tuple_gen_fn}
\OL{B}_k(q)=\sum_{n\geq0}\OL{b}_k(n)q^n = \frac{f^k_4}{f^{2k}_1f^k_2}
\end{equation}
with $\OL{b}_0(n) = 0$ for all $n\geq 1$.

The above functions have received recent attention for small values of $k$.  For example, many congruences for the function $\OL{b}_2(n)$ \cite{ByungchanKimOvercubicPairs, ChiranjitRay, ShivashankarOvercubicPairs} and the function $\OL{b}_3(n)$ \cite{JiayuChen, TheOG, ManjilSaikia} have appeared in the literature.  Most of these congruences have moduli which are small powers of 2.  At the time of this writing, only one paper appears to have been published which deals with the function $\OL{b}_k(n)$ for arbitrary $k$, which is the work of Buragohain and Saikia \cite{PujashreeBuragohain}.  In \cite{PujashreeBuragohain}, the focus is on congruences with moduli which are powers of 2.  As part of the work that we share below, we revisit such results from a new perspective by rewriting the generating functions in question in terms of Ramanujan's theta function $\varphi(q)$. 

Here are a few congruence results which have appeared in the literature recently for the case $k=3$ \cite{JiayuChen, TheOG, ManjilSaikia}.  

\begin{theorem}
For all $m, \alpha \geq 0$,
    \begin{align}
        \OL{b}_3(128m+80) &\equiv 0 \pmod{4},\\
        \OL{b}_3(256\cdot 2^\alpha m+192) &\equiv 0 \pmod{4},\\
        \OL{b}_3(32m+20) &\equiv 0 \pmod{32}, \\
        \OL{b}_3(32m+28) &\equiv 0 \pmod{64}, \text{ and} \\
        \OL{b}_3(2^\alpha(8m+7)) &\equiv 0 \pmod{128}.
    \end{align}
\end{theorem}

With the above in mind, our main interest is to study arithmetic properties satisfied by $\OL{b}_k(n)$ for a variety of values of $k$.  In particular, we wish to provide a new, unified perspective on the various congruences modulo small powers of 2 satisfied by $\OL{b}_k(n)$ which appear in the literature, and we will also prove several infinite families of congruences satisfied by these functions for moduli which are not powers of 2.  (We note in passing that, for $k\geq 4$, there are no congruences which appear in the literature for the functions $\OL{b}_k(n)$ with moduli which are not powers of 2.)   
In Section \ref{sec:preliminaries}, we will provide a list of all of the $q$-series results that we need to prove our results.  In Section \ref{sec:mod2Perspective}, we will discuss our new perspective on the various congruences modulo powers of 2 that have been proven recently for this family of functions, while in Section \ref{sec:OddModuli}, we will prove a number of congruences with moduli which are odd.  We share some concluding remarks in Section \ref{sec:concluding_thoughts}.  All of the proof techniques utilized below are elementary, relying on well--known $q$-series results and standard dissection techniques for manipulating generating functions.

\section{Preliminaries}
\label{sec:preliminaries}

As noted above, we will collect all of our required tools in this section.  We begin with one of the most classical of all the $q$-series results in this field.  

\begin{theorem}[Euler's Pentagonal Number Theorem]
\label{Euler_pentagonal_number_thm}
    We have
    $$f_1 = (q;q)_\infty =\sum_{k=-\infty}^\infty(-1)^kq^{\frac{3k^2 + k}{2}}.$$
    \end{theorem}
\begin{proof}
See Hirschhorn \cite[(1.6.1)]{Powerofq}.    
\end{proof}

Next, we will need to understand what happens in Theorem \ref{Euler_pentagonal_number_thm} when $q$ is replaced by $-q$.  The following lemma provides the answer.  
\begin{lemma}\label{Negative q}
    We have
    $$(-q;-q)_\infty = \frac{f_2^3}{f_1f_4}.$$
\end{lemma}
\begin{proof}
Using elementary generating function manipulations, we have
\begin{align*}
(-q;-q)_\infty 
&=
(q^2;q^2)_\infty(-q;q^2)_\infty \\
&=
f_2\cdot \frac{(q^2;q^4)_\infty}{(q;q^2)_\infty} \\
&=
f_2\cdot \frac{(q^2;q^2)_\infty}{(q^4;q^4)_\infty}\cdot \frac{(q^2;q^2)_\infty}{(q;q)_\infty}\\
&=
\frac{f_2^3}{f_1f_4}.
\end{align*}
\end{proof}
In the same vein, it will be beneficial in the work below to consider the generating function $\OL{B}_k(q)$ when $q$ is replaced by $-q$.  
\begin{corollary}\label{Negative Overcubic}
    We have
    $$\OL{B}_k(-q)=\sum_{n\geq0}\OL{b}_k(n)(-q)^n = \frac{f_1^{2k}f_4^{3k}}{f_2^{7k}}.$$
    \begin{proof}
        Beginning with the definition of $\OL{B}_k(q)$, we have
        \begin{align*}
            \OL{B}_k(-q)= \sum_{n\geq0}\OL{b}_k(n)(-q)^n &= \infprod \frac{(1-(-q)^{4i})^k}{(1-(-q)^i)^{2k}(1-(-q)^{2i})^k} \\
            &= \infprod\frac{(1-q^{4i})^k}{(1-q^{2i})^k}\infprod \frac{1}{(1-(-q)^i)^{2k}} \\
            &= \frac{f^k_4}{f^k_2} \left(\infprod (1-(-q)^i)\right)^{-2k}.
        \end{align*}
        Using Lemma \ref{Negative q}, we obtain
        %\begin{align*}
         $$   \sum_{n\geq0}\OL{b}_k(n)(-q)^n =   \frac{f^k_4}{f^k_2} \left(\frac{f_2^3}{f_1f_4}\right)^{-2k} 
            =   \frac{f^k_4}{f^k_2} \frac{f_1^{2k}f_4^{2k}}{f_2^{6k}}
            =   \frac{f_1^{2k}f_4^{3k}}{f_2^{7k}}.
        $$
        %\end{align*}
    \end{proof}
\end{corollary}

Next, we highlight the $q$-series $P(q)$ which is defined as 
    $$P(q) = \sum^\infty_{k=-\infty} q^{\frac{3k^2+k}{2}}.$$
 
\begin{theorem}\label{P(q) Gen Fn}
    We have
    $$P(q) = \frac{f_2f^2_3}{f_1f_6}.$$
\end{theorem}

\begin{proof}
See Hirschhorn \cite[(14.3.3) and (26.1.2)]{Powerofq}. 
\end{proof}

\begin{corollary}\label{Negative P(q) Gen Fn}
    We have
    $$P(-q) = \frac{f_1f_4f_6^5}{f_2^2f_3^2f_{12}^2}.$$
\end{corollary}
\begin{proof}
Using Lemma \ref{Negative q} and Theorem \ref{P(q) Gen Fn}, we have 
\begin{align*}
P(-q)
&= 
\frac{f_2(-q^3;-q^3)_\infty^2}{f_6(-q;-q)_\infty} \\
&= 
\frac{f_2}{f_6}\cdot \left( \frac{f_6^3}{f_3f_{12}}\right)^2\cdot \frac{f_1f_4}{f_2^3} \\
&= 
\frac{f_1f_2f_4f_6^6}{f_2^3f_3^2f_6f_{12}^2} \\
&= 
\frac{f_1f_4f_6^5}{f_2^2f_3^2f_{12}^2}
\end{align*}
after simplification. 
\end{proof}

We will also utilize several results related to Ramanujan's theta functions $\varphi(q)$ and $\psi(q)$. 

% The mega theorem

\begin{theorem}
    We have 
    \begin{align}
        \varphi(q) &= \frac{f^5_2}{f_1^2f_4^2}, \label{Phi Gen Fn} \\
        \varphi(-q) &= \frac{f_1^2}{f_2}, \label{Negative Phi Gen Fn} \\
%        \varphi(q) &= \varphi(q^4) + 2q\psi(q^8), \label{2-Dissection Phi} \\
        \psi(q) &= \frac{f^2_2}{f_1},\label{Psi Gen Fn} \\
        \psi(-q) &= \frac{f_1f_4}{f_2}, \label{Negative Psi Gen Fn} \\
        \psi(q) &= P(q^3) + q\psi(q^9), \label{3-Dissection Psi} \\
        \frac{1}{\psi(q)} &= \frac{\psi(q^9)}{\psi(q^3)^4}\left(P(q^3)^2-qP(q^3)\psi(q^9)+q^2\psi(q^9)^2\right), \text{ and }  \label{3-Dissection Reciprocal Psi} \\
        \psi(q)f_4^3 &= f_5^3\varphi(-q^{10})-5q^{15}f_{100}^3\psi(q^{25}) + E_1 + E_2 + E_3 + E_4 \label{5-Dissection}
    \end{align}
    where $E_i$ is a $q$-series where the powers of $q$ are of the form $q^{5j+i}$.
\end{theorem}
\begin{proof}
    For Equations \eqref{Phi Gen Fn}, \eqref{Psi Gen Fn}, and \eqref{3-Dissection Psi}, see equations (1.5.6), (1.5.7), and (14.3.3) of Hirschhorn \cite{Powerofq}, respectively.
    Equations \eqref{Negative Phi Gen Fn} and \eqref{Negative Psi Gen Fn} are proven using Lemma \ref{Negative q}.
    Lastly, for Equation \eqref{3-Dissection Reciprocal Psi}, see Hirschhorn and Sellers \cite[Lemma 2.2]{OddPartsDistinct}, and for Equation \eqref{5-Dissection}, see Hirschhorn and Sellers \cite[Lemma 2.2]{4colorFrobenius}.
\end{proof}

Beyond the results mentioned above, we will require a few more $q$-series results.  
We begin with dissections related to the cubic partition function $b(n)$.  In order to state the results, we define $F(q) = f_1f_2$.

\begin{theorem}\label{Capital F Dissection}
    We have
    $$F(q) = F(q^9)\left(X(q^3)^{-1}-q-2q^2X(q^3)\right) \text{ where } 
    X(q) = \frac{f_1f_6^3}{f_2f_3^3}.$$
\end{theorem}
\begin{proof}
See Hirschhorn \cite[(14.3.1)]{Powerofq}.    
\end{proof}

\begin{theorem}\label{cubic 3-dissection}
We have
$$B(q) = \frac{1}{F(q)} = \frac{F(q^9)^3}{F(q^3)^4}\left(X(q^3)^{-2}+qX(q^3)^{-1} +3q^2-2q^3X(q^3)+4q^4X(q^3)^2\right).$$
\end{theorem}
\begin{proof}
See Hirschhorn \cite[(14.4.5)]{Powerofq}.    
\end{proof}

The number of partitions of $n$ where even parts are distinct is denoted as $ped(n)$ and the generating function appears in the work of Andrews, Hirschhorn, and Sellers \cite{EvenPartsDistinct}.

\begin{theorem}\label{even parts distinct partition}
    We have
    $$\sum_{n\geq0}ped(n)q^n = \frac{f_4}{f_1}.$$
\end{theorem}
Furthermore, Andrews, Hirschhorn, and Sellers provide the following 3-dissection.  

\begin{theorem}\label{ped dissection}
    We have
    $$\sum_{n\geq0}ped(n)q^n = \frac{f_{12}f_{18}^4}{f_3^3f_{36}^2} + q\frac{\varphi(-q^9)\psi(-q^9)}{\varphi(-q^3)^2}+2q^2\frac{f_6f_{18}f_{36}}{f_3^3}.$$
\end{theorem}
\begin{proof}
See \cite[Theorem 3.1]{EvenPartsDistinct}.    
\end{proof}

One additional 2-dissection result is necessary below.  
\begin{theorem}\label{reciprocal of f_1^4 2-dissection}
    We have $$\frac{1}{f_1^4} = \frac{f^{14}_4}{f^{14}_2f^4_8} +4q\frac{f_4^2f_8^4}{f_2^{10}}.$$
\end{theorem}
\begin{proof}
See Brietzke, da Silva, and Sellers \cite[Lemma 2.2]{Brietzke}.

\end{proof}

Our last $q$-series identity dates back to Ramanujan.   
\begin{theorem}\label{funPentagonalNumberResults}
    We have
    $$\frac{f_1^5}{f_2^2} = \sum_{n=-\infty}^\infty(6n+1)q^\frac{3n^2+n}{2}.$$
\end{theorem}
\begin{proof}
See Hirschhorn \cite[(10.7.3)]{Powerofq}. 
\end{proof}
Finally, in order to quickly manipulate our generating functions to prove the congruences in question, we will rely heavily on the following lemma. 
\begin{lemma}
\label{Expanded Freshman's Gn Fn}
    For prime $p$ and integers $\alpha\geq 0$ and $l\geq1$, we have
    \begin{align*}
    f_{pl}^{p^{\alpha}} &\equiv f_l^{p^{\alpha+1}} \pmod{p^{\alpha+1}} \text{\ \ and}\\
    \frac{1}{f_{pl}^{p^{\alpha}}} &\equiv \frac{1}{f_l^{p^{\alpha+1}}} \pmod{p^{\alpha+1}}.
    \end{align*}
\end{lemma}
\begin{proof}
The first congruence above follows from the Binomial Theorem and divisibility properties of binomial coefficients.  The second congruence follows from the first once we note that, for all $r\geq 1$, $f_r$ has constant term 1 and is, therefore, invertible.  
\end{proof}
\section{A Different Perspective on Congruences Modulo Powers of 2}
\label{sec:mod2Perspective}
We begin this section by noting the following functional equation for $\OL{B}_k(q)$ which serves as a generalization of Theorem \ref{b_1 phi recurrence}.  
\begin{theorem}\label{b_k phi Recurrence}
    For all $k \geq 1$, $\OL{B}_k(q) = \varphi^k(q)\varphi^k(q^2)\left(\OL{B}_k(q^2)\right)^2$.
\end{theorem}  
\begin{corollary}\label{B_k Prod Defn}
    For all $k\geq 1$, $\OL{B}_k(q) = \varphi^k(q)\infprod \left(\varphi(q^{2^i})\right)^{3k\cdot2^{i-1}}$.
\end{corollary}

Note that Theorem \ref{b_k phi Recurrence} and Corollary \ref{B_k Prod Defn}
immediately follow from Theorem \ref{b_1 phi recurrence} and Corollary \ref{b_1 phi defn} by raising both sides of the corresponding equations to the power $k$.

A priori, it should be clear that Corollary \ref{B_k Prod Defn} provides an extremely beneficial perspective on the generating function of $\OL{b}_k(n)$.  Simply put, given that 
$$\varphi(q) = 1 + 2\sum_{n\geq1}q^{n^2},$$
we must expect that $\OL{b}_k(n)$ will likely satisfy a wide variety of congruences modulo powers of 2.  This is true because, thanks to the Binomial Theorem, for a fixed positive integer $t$, we have 
$$
(\varphi(q))^t = \left(1 + 2\sum_{n\geq1}q^{n^2}\right)^t = \sum_{j=0}^t \binom{t}{j}2^j \left(\sum_{n\geq1}q^{n^2}\right)^j
$$
and the factors of the form $2^j$ will vanish modulo small powers of 2, providing a great deal of simplification when attempting to prove congruences satisfied by $\OL{b}_k(n)$ modulo a small power of 2.  

For example, with Corollary \ref{B_k Prod Defn} in hand, we can now greatly extend Theorem \ref{b_1 mod 4 characterization} which characterized the number of overcubic partitions modulo 4.

\begin{theorem}\label{Mod 4 Characterization}
We have the following:  
\begin{itemize}
\item{}    For all $n \geq 1$ and $k$ even, $\OL{b}_k(n) \equiv 0 \pmod{4}$. 
\smallskip  
\item{}   For all $n \geq 1$ and $k$ odd,
            \[ \OL{b}_k(n) \equiv \begin{cases}
                    2 \pmod{4} \text{ if } n = l^2 \text{ or } n = 2l^2 \text{, and} \\
                    0 \pmod{4} \text{ otherwise.}
                    \end{cases} \]
\end{itemize}
    \begin{proof}        
    Thanks to Equation \eqref{phi_sum_defn} and Corollary \ref{B_k Prod Defn}, we have 
        \begin{align*}
            \OL{B}_k(q) = \sum_{n\geq0}\OL{b}_k(n)q^n &= \varphi^k(q)\varphi^{3k}(q^2)\prod_{i\geq2}^{\infty} \left(\varphi(q^{2i})\right)^{3k\cdot2^{i-1}} \\
            &\equiv \varphi^k(q)\varphi^{3k}(q^2) \pmod{4} \\
            &\equiv \left(1 + 2\sum_{l\geq1}q^{l^2}\right)^k\left(1 + 2\sum_{l\geq1}q^{2l^2}\right)^{k\cdot3}  \pmod{4}.
        \end{align*}
        Consider the case of $k$ even.
        Then, $k =  2j$ for some integer $j \geq 1$.
        Now, our congruence becomes
        \begin{align*}
            \sum_{n\geq0}\OL{b}_{2j}(n)q^n
            &\equiv \left(1 + 2\sum_{l\geq1}q^{l^2}\right)^{2\cdot j}\left(1 + 2\sum_{l\geq1}q^{2l^2}\right)^{2\cdot3j}  \pmod{4} \\
            &\equiv \left(1 + 4\sum_{l\geq1}q^{l^2} + 4\left(\sum_{l\geq1}q^{l^2}\right)^2\right)^{j}\left(1 + 4\sum_{l\geq1}q^{2l^2} + 4\left(\sum_{l\geq1}q^{2l^2}\right)^2\right)^{3j}  \pmod{4} \\
            &\equiv 1 \pmod{4}.
        \end{align*}
        This yields our first result.
        
        Next, consider the case of $k$ odd.
        Then, $k =  2j+1$ for some integer $j \geq 0$.
        Using this and our work above, we know
        \begin{align*}
            \sum_{n\geq0}\OL{b}_{2j+1}(n)q^n
            &\equiv \left(1 + 2\sum_{l\geq1}q^{l^2}\right)^{2\cdot j+1}\left(1 + 2\sum_{l\geq1}q^{2l^2}\right)^{(2\cdot j+1)\cdot3}  \pmod{4} \\
            &\equiv \left(1 + 2\sum_{l\geq1}q^{l^2}\right)^{2\cdot j}\left(1 + 2\sum_{l\geq1}q^{l^2}\right)\left(1 + 2\sum_{l\geq1}q^{2l^2}\right)^{2j\cdot3+2}\left(1 + 2\sum_{l\geq1}q^{2l^2}\right)  \pmod{4} \\
            &\equiv \left(1 + 2\sum_{l\geq1}q^{l^2}\right)\left(1 + 2\sum_{l\geq1}q^{2l^2}\right)  \pmod{4} \\
            &= 1 + 2\sum_{l\geq1}q^{l^2} + 2\sum_{l\geq1}q^{2l^2} + 4\left(\sum_{l\geq1}q^{l^2} \right)\left(\sum_{l\geq1}q^{2l^2}\right)\\
            &\equiv  1 + 2\sum_{l\geq1}q^{l^2} + 2\sum_{l\geq1}q^{2l^2}  \pmod{4}.
        \end{align*}
This yields our second result.  
    \end{proof}
\end{theorem}

%As an aside, one could easily imagine using the above proof technique to generalize Theorem \ref{b_1 mod 8 characterization} as well.  Admittedly, the $q$-series manipulations would become slightly more complicated than what appears above for the mod 4 characterization; nevertheless, this could be done.  We leave the details to the interested reader.  

Given this new perspective on finding characterizations modulo small powers of 2 for the functions $\OL{b}_k(n)$, we now wish to revisit some of the results that already exist in the literature with the goal of providing very elementary (and short) proofs.  

As an example, Shivaprasada Nayaka et al. \cite{TheOG} note many congruences modulo 4 satisfied by $\OL{b}_3(n)$.  
Using Theorem \ref{Mod 4 Characterization}, we can verify some of their results by way of example.

\begin{theorem}[\cite{TheOG}] For all $m, \alpha \geq 0$,
    \begin{align*}
        \OL{b}_3(64m+48) &\equiv 0 \pmod{4},\\
        \OL{b}_3(128m+80) &\equiv 0 \pmod{4}, \text{ and}\\
        \OL{b}_3(256\cdot 2^\alpha m+192) &\equiv 0 \pmod{4}.
    \end{align*}
    \begin{proof}
        We provide a proof which relies on Theorem \ref{Mod 4 Characterization} and is very different in nature to the proofs used in \cite{TheOG}.  Using Theorem \ref{Mod 4 Characterization}, we only need to prove that the arithmetic progressions mentioned above can never contain a square or twice a square.
        \begin{itemize}
            \item We have $64m + 48 = 16(4m+3)$.
            Now, $4m+3 \equiv 3 \pmod{4}$, hence it cannot be square.
            Therefore, no number of the form $16(4m+3) = 64m + 48$ can be square.
            Also, $64m + 48 = 2\cdot4(8m+6)$.
            Note that $8m + 6 \equiv 2 \pmod{4}$.
            Therefore, no number of the form $64m + 48$ can be twice a square.
            \item We have $128m + 80 = 16(8m+5)$.
            We know $8m+5 \equiv 5 \pmod{8}$, and the whole expression cannot be a square.
            Also, $128m + 80 = 2\cdot4(16m+10)$.
            Note that $16m + 10 \equiv 2 \pmod{8}$, and the whole expression cannot be twice a square.
            \item We have $256\cdot 2^\alpha m+192 = 64(4\cdot2^\alpha m+3)$.
            We know $4\cdot2^\alpha m+3 \equiv 3 \pmod{4}$, so the whole expression cannot be a square.
            Also $256\cdot 2^\alpha m+192 = 2\cdot16(8\cdot2^\alpha m+6)$.
            Note that $8\cdot2^\alpha m+6 \equiv 2 \pmod{4}$, so the whole expression cannot be twice a square.            
        \end{itemize}
Therefore, by Theorem \ref{Mod 4 Characterization}, all of the congruences mentioned above must hold.
    \end{proof}
\end{theorem}
Note that the specific value $k=3$ does not play a role in the above proof (other than the fact that 3 is odd).  Therefore, we see that the above theorem can be extended immediately to an infinite family of congruences modulo 4.  

\begin{theorem} For all $k, m, \alpha \geq 0$,
    \begin{align*}
        \OL{b}_{2k+1}(64m+48) &\equiv 0 \pmod{4},\\
        \OL{b}_{2k+1}(128m+80) &\equiv 0 \pmod{4}, \text{ and}\\
        \OL{b}_{2k+1}(256\cdot 2^\alpha m+192) &\equiv 0 \pmod{4}.
    \end{align*}
\end{theorem}   

Of course, the above theorem can be generalized even further to include other arithmetic progressions, as long as such arithmetic progressions never contain a square or twice a square.  We stop here for the sake of brevity.  

\section{Congruences Involving Odd Moduli}
\label{sec:OddModuli}
We now transition to congruences satisfied by $\OL{b}_k(n)$ for various values of $k$ where the moduli in question are odd.  
We begin by focusing on a few results for specific values of $k$.  

\begin{theorem}
    For all $m \geq 0$, $\OL{b}_{4}(14m + 7) \equiv 0 \pmod{7}$.
    \begin{proof}
        We have
        $$\OL{B}_4(q)= \sum_{n\geq0}\OL{b}_{4}(n)q^{n} = \frac{f_4^4}{f_1^{8}f_2^4} = \frac{f^4_4}{f^4_2}\frac{1}{f^8_1}= \frac{f^4_4}{f^4_2}\left(\frac{1}{f^4_1}\right)^2.$$
        Using Theorem \ref{reciprocal of f_1^4 2-dissection}, our equation becomes
        $$
            \sum_{n\geq0}\OL{b}_{4}(n)q^{n} = \frac{f^4_4}{f^4_2}\left( \frac{f^{14}_4}{f^{14}_2f^4_8} +4q\frac{f_4^2f_8^4}{f_2^{10}} \right)^2 
            = \frac{f^4_4}{f^4_2}\left( \frac{f^{28}_4}{f^{28}_2f^8_8} +  8q\frac{f^{14}_4}{f^{14}_2f^4_8}\frac{f_4^2f_8^4}{f_2^{10}}+16q^2\frac{f_4^4f_8^8}{f_2^{20}} \right).
        $$
        Considering the odd powers of $q$ yields
        \begin{align*}
            \sum_{n\geq0}\OL{b}_{4}(2n+1)q^{2n+1} &=\frac{f^4_4}
            {f^4_2}\left( 8q\frac{f^{14}_4}{f^{14}_2f^4_8}\frac{f_4^2f_8^4}{f_2^{10}} \right)
            = 8q \frac{f_4^{20}}{f_2^{28}}.
        \end{align*}
        Dividing both sides by $q$ and substituting $q$ for $q^2$, our generating function becomes
        \begin{align*}
            \sum_{n\geq0}\OL{b}_{4}(2n+1)q^{n} &= 8 \frac{f_2^{20}}{f_1^{28}} = 8\frac{f_2^{21}}{f_1^{28}f_2} \\
            &\equiv \frac{f_{14}^3}{f_7^4}\frac{1}{f_2} \pmod{7}\\
            &= \frac{f_{14}^3}{f_7^4}\sum_{N\geq0} p(N)q^{2N}.
        \end{align*}
        Note that $\frac{f_{14}^3}{f_7^4}$ is a function of $q^7$.  Thus, when the last expression above is written as a power series in $q$, the powers of $q$ are of the form $7N' + 2N$ for some nonnegative integers $N'$ and $N$.
        Our goal remains to find $\OL{b}_4(14m+7)$, and substituting $7m+3$ for $n$ above creates such an argument within our function.
        Completing this substitution, and then comparing the powers of $q$ on both sides, we have
        \[
        \begin{aligned}[t]
            && 7m  + 3 &= 7N' + 2N & \\
            \Longrightarrow && 3 &\equiv 2N && \pmod{7} \\
            \Longrightarrow && 4(3) &\equiv 4(2N) && \pmod{7} \\
            \Longrightarrow && 5 &\equiv N && \pmod{7}.
        \end{aligned}
        \]    
        Thus, $N = 7j + 5$ for some integer $j\geq 0$. This means
        $$\sum_{m\geq0}\OL{b}_{4}(14m+7)q^{7m+3} \equiv \frac{f_{14}^3}{f_7^4}\sum p(7j+5)q^{14j+10} \pmod{7}.$$
        By Theorem \ref{ramanujanCongruences}, we know $p(7j+5) \equiv 0 \pmod{7}$  for all $j\geq0$.  Thus, since $\frac{f_{14}^3}{f_7^4}$ is a function of $q^7$, our result holds.  
    \end{proof}
\end{theorem}

We can prove a similar result for $\OL{b}_4(n)$ modulo 11.
    
\begin{theorem}
    For all $m\geq0$, $\OL{b}_{4}(22m + 11) \equiv 0 \pmod{11}$.
    \begin{proof}
        Using the work in the proof of the previous theorem, we know
        $$\sum_{n\geq0}\OL{b}_{4}(2n+1)q^{2n+1} = 8q \frac{f_4^{20}}{f_2^{28}} = 8q \frac{f_4^{22}}{f_2^{33}}\frac{f_2^5}{f_4^2}.$$
        Using Theorem \ref{funPentagonalNumberResults}, our equation becomes
        $$\sum_{n\geq0}\OL{b}_{4}(2n+1)q^{2n+1} =8q \frac{f_4^{22}}{f_2^{33}}\left(\sum_{N=-\infty}^\infty (6N+1)q^{3N^2+N}\right).$$
        Dividing both sides by $q$ and substituting $q$ for $q^2$, we have
         $$
            \sum_{n\geq0}\OL{b}_{4}(2n+1)q^n = 8\frac{f_2^{22}}{f_1^{33}}\left(\sum_{N=-\infty}^\infty (6N+1)q^{\frac{3N^2+N}{2}}\right) 
         \equiv 8\frac{f_{22}^{2}}{f_{11}^{3}}\left(\sum_{N=-\infty}^\infty (6N+1)q^{\frac{3N^2+N}{2}}\right) \pmod{11}.
        $$
        Note that $\frac{f_{22}^2}{f_{11}^3}$ is a function of $q^{11}$.  Thus, when the last expression above is written as a power series in $q$, the powers of $q$ are of the form $11N' + \frac{3N^2+N}{2}$ for some nonnegative integers $N'$ and $N$.
        Our goal remains to find $\OL{b}_4(22m+11)$, and substituting $11m+5$ for $n$ above creates such a pattern within our function.
        Completing this substitution, and then comparing the powers of $q$ on both sides of the congruence, we have
        \[
        \begin{aligned}[t]
            && 11m + 5 &= 11N' + \frac{3N^2+N}{2} & \\
            \Longrightarrow && 5 &\equiv \frac{3N^2+N}{2} && \pmod{11} \\
            \Longrightarrow && 24\cdot5 &\equiv 24\left(\frac{3N^2+N}{2}\right) && \pmod{11} \\
            \Longrightarrow && 120+1 &\equiv 36N^2+12N+1 && \pmod{11} \\
            \Longrightarrow && 121 &\equiv (6N+1)^2 && \pmod{11} \\
            \Longrightarrow && 0 &\equiv (6N+1)^2 && \pmod{11}.
        \end{aligned}
        \]    
        This means that, within the arithmetic progression of interest to us, $6N+1 \equiv 0 \pmod{11}$.  Since $\frac{f_{22}^{2}}{f_{11}^{3}}$ is a function of $q^{11}$, our result follows.  
    \end{proof}
\end{theorem}
Next, we have a congruence modulo 5 satisfied by $\OL{b}_8(n)$.
\begin{theorem}\label{b8(5m)}
    For all $m\geq1$, $\OL{b}_{8}(5m) \equiv 0 \pmod{5}$.
    % from 2025-11-10
    \begin{proof}
        We know
$$
\sum_{n\geq0}\OL{b}_{8}(n)q^{n} = \frac{f_4^8}{f_1^{16}f_2^8}
            =\frac{f_4^{5}}{f_1^{15}f_2^{10}}\frac{f_2^2f_4^3}{f_1}
            =\frac{f_4^{5}}{f_1^{15}f_2^{10}}\frac{f_2^2}{f_1}f_4^3
            =\frac{f_4^{5}}{f_1^{15}f_2^{10}}\psi(q)f_4^3
            \equiv \frac{f_{20}}{f_5^{3}f_{10}^{2}}\psi(q)f_4^3 \pmod{5}.
$$
        Using Equation \eqref{5-Dissection},  and considering only the terms of our $q$-series of the form $q^{5n}$, our congruence becomes
        $$
        \sum_{n\geq0}\OL{b}_{8}(5n)q^{5n} 
        %= \frac{f_4^{5}}{f_1^{15}f_2^{10}}\left(f_5^3\varphi(-q^{10})-5q^{15}f_{100}^3\psi(q^{25})\right)
        \equiv \frac{f_{20}}{f_5^{3}f_{10}^{2}}f_5^3\varphi(-q^{10}) \pmod{5}.
        $$
        Now, we substitute $q$ for $q^5$ to get
        \begin{align*}
        \sum_{n\geq0}\OL{b}_{8}(5n)q^{n} &\equiv \frac{f_{4}}{f_1^{3}f_{2}^{2}}f_1^3\varphi(-q^{2}) \pmod{5}\\
        &\equiv \frac{f_{4}}{f_{2}^{2}}\frac{f_2^2}{f_4} \pmod{5} \text{ from Equation \eqref{Negative Phi Gen Fn}}\\
        &\equiv 1 \pmod{5} \\
        &\equiv 1 +0q+0q^2+0q^3+0q^4+\dots \pmod{5}.
        \end{align*}
        Therefore, for all $m\geq1$, $\OL{b}_{8}(5m) \equiv 0 \pmod{5}$.
    \end{proof}
\end{theorem}

Interestingly enough, Theorem \ref{b8(5m)} can be used to prove an infinite family of congruences modulo 5.  

\begin{theorem}
\label{b8(5m)generalized}    
For all $j\geq1$, $m\geq 0$, and all $1\leq r\leq 4$, $\OL{b}_{25j+8}(5(5m+r)) \equiv 0 \pmod{5}$.
\end{theorem}
\begin{proof}
Note that 
\begin{align*}
\sum_{n\geq0}\OL{b}_{25j+8}(n)q^{n} 
&= \frac{f_4^{25j+8}}{f_1^{2(25j+8)}f_2^{25j+8}}   \\        
&=\frac{f_4^{25j}}{f_1^{50j}f_2^{25j}}\cdot \frac{f_4^{8}}{f_1^{16}f_2^{8}} \\ 
&= \frac{f_4^{25j}}{f_1^{50j}f_2^{25j}}\sum_{n\geq0}\OL{b}_{8}(n)q^{n} \\
&\equiv 
\frac{f_{100}^{j}}{f_{25}^{2j}f_{50}^{j}}\sum_{n\geq0}\OL{b}_{8}(n)q^{n} \pmod{5}.
\end{align*}
Thus, we see that 
\begin{align*}
\sum_{n\geq0}\OL{b}_{25j+8}(5n)q^{n} 
&\equiv 
\frac{f_{20}^{j}}{f_{5}^{2j}f_{10}^{j}}\sum_{n\geq0}\OL{b}_{8}(5n)q^{n} \pmod{5} \\
&\equiv \frac{f_{20}^{j}}{f_{5}^{2j}f_{10}^{j}}\cdot 1 \pmod{5} 
\end{align*}
thanks to the proof of Theorem \ref{b8(5m)}.  Since $\frac{f_{20}^{j}}{f_{5}^{2j}f_{10}^{j}}$ is a function of $q^5$, we know that, for any $n$ which is not divisible by 5, we must have
$\OL{b}_{25j+8}(5n) \equiv 0 \pmod{5}$.  The theorem follows.  
\end{proof}

We next move to a set of divisibility properties modulo 3. We begin with a pair of congruences modulo 3 satisfied by $\OL{b}_5(n)$.

\begin{theorem}\label{b5(27m+9)}
For all $m\geq0$, $\OL{b}_5(27m+9) \equiv \OL{b}_5(27m+18) \equiv 0 \pmod{3}$.
    % from 2025-11-03 and proof method from 2025-10-22
    \begin{proof}
        Note that replacing $q$ by $-q$ in a partition function will not change its divisibility properties.
        Using Corollary \ref{Negative Overcubic} and Equation \eqref{Psi Gen Fn}, we have
       $$
            \sum_{n\geq0}\OL{b}_5(n)(-q)^n = \frac{f^{10}_1f^{15}_4}{f^{35}_2} 
            = \frac{f^{9}_1f^{15}_4}{f^{33}_2}\frac{f_1}{f^{2}_2} 
            \equiv \frac{f^{3}_3f^{5}_{12}}{f^{11}_6}\frac{1}{\psi(q)} \pmod{3}.
       $$
        Using Equation \eqref{3-Dissection Reciprocal Psi}, we know
        $$\sum_{n\geq0}\OL{b}_5(n)(-q)^n \equiv \frac{f^{3}_3f^{5}_{12}}{f^{11}_6}\left(\frac{\psi(q^9)}{\psi(q^3)^4}\left(P(q^3)^2-qP(q^3)\psi(q^9)+q^2\psi(q^9)^2\right) \right) \pmod{3}.$$
        Define $B'_5(n) = (-1)^n\OL{b}_5(n)$ so that
        $$\sum_{n\geq0}B'_5(n)q^n \equiv \frac{f^{3}_3f^{5}_{12}}{f^{11}_6}\left(\frac{\psi(q^9)}{\psi(q^3)^4}\left(P(q^3)^2-qP(q^3)\psi(q^9)+q^2\psi(q^9)^2\right) \right) \pmod{3}.$$
        Considering the terms of our $q$-series of the form $q^{3n}$, our congruence becomes
        $$\sum_{n\geq0}B'_5(3n)q^{3n} \equiv \frac{f^{3}_3f^{5}_{12}}{f^{11}_6}\frac{\psi(q^9)}{\psi(q^3)^4}P(q^3)^2 \pmod{3}.$$
        Substituting $q$ for $q^3$, we get        
        \begin{align*}
            \sum_{n\geq0}B'_5(3n)q^n
            &\equiv \frac{f^{3}_1f^{5}_{4}}{f^{11}_2}\frac{\psi(q^3)}{\psi(q)^4}P(q)^2 \pmod{3} \\
            &\equiv \frac{f^{3}_1f^{6}_{4}}{f^{12}_2}\frac{f_2}{f_4}\frac{\psi(q^3)}{\psi(q^3)\psi(q)}P(q)^2 \pmod{3} \\
            &\equiv \frac{f_3f^{2}_{12}}{f^{4}_6}\frac{f_2}{f_4}\frac{f_1}{f^2_2}P(q)^2 \pmod{3} \\
            &\equiv \frac{f_3f^{2}_{12}}{f^{4}_6}\frac{f_1}{f_2f_4}P(q)^2 \pmod{3}.
        \end{align*}
        Using Theorem \ref{P(q) Gen Fn}, we have
        \begin{align*}
            \sum_{n\geq0}B'_5(3n)q^n 
            &\equiv \frac{f_3f^{2}_{12}}{f^{4}_6}\frac{f_1}{f_2f_4}\left(\frac{f_2f^2_3}{f_1f_{6}}\right)^2 \pmod{3} \\
            &\equiv \frac{f_3f^{2}_{12}}{f^{4}_6}\frac{f_1}{f_2f_4}\frac{f^2_2f^4_3}{f^2_1f^2_{6}} \pmod{3} \\
            &\equiv \frac{f^5_3f^{2}_{12}}{f^{6}_6}\frac{f_2}{f_1f_4} \pmod{3} \\
            &\equiv \frac{f^5_3f^{2}_{12}}{f^{6}_6}\frac{1}{\psi(-q)} \pmod{3}
        \end{align*}
        using Equation \eqref{Negative Psi Gen Fn}.
        Using the 3-dissection from Equation \eqref{3-Dissection Reciprocal Psi} and substituting $-q$ for $q$, we know
        $$ \sum_{n\geq0}B'_5(3n)q^n \equiv \frac{f^5_3f^{2}_{12}}{f^{6}_6}\left(\frac{\psi(-q^9)}{\psi(-q^3)^4}\left(P(-q^3)^2+qP(-q^3)\psi(-q^9)+q^2\psi(-q^9)^2\right)\right) \pmod{3}.$$
        Considering the terms of the above $q$-series of the form $q^{3n}$, our congruence becomes
        $$ \sum_{n\geq0}B'_5(9n)q^{3n}\equiv \frac{f^5_3f^{2}_{12}}{f^{6}_6}\frac{\psi(-q^9)}{\psi(-q^3)^4}P(-q^3)^2 \pmod{3}.$$
        Substituting $q$ for $q^3$ again, we get
        $$\sum_{n\geq0}B'_5(9n)q^{n}  \equiv \frac{f^5_1f^{2}_{4}}{f^{6}_2}\frac{\psi(-q^3)}{\psi(-q)^4}P(-q)^2 \pmod{3}.$$
        Applying Corollary \ref{Negative P(q) Gen Fn}, our congruence becomes
        \begin{align*}
            \sum_{n\geq0}B'_5(9n)q^{n} &\equiv \frac{f^5_1f^{2}_{4}}{f^{6}_2}\frac{\psi(-q)^3}{\psi(-q)^4}\left(\frac{f_1f_4f^5_6}{f^2_2f^2_3f^2_{12}}\right)^2 \pmod{3} \\
            &\equiv \frac{f^5_1f^{2}_{4}}{f^{6}_2}\frac{1}{\psi(-q)}\frac{f_1^2f_4^2f_6^{10}}{f_2^4f_3^4f_{12}^4}\pmod{3} \\
            &\equiv \frac{f^5_1f^{2}_{4}}{f^{6}_2}\frac{f_2}{f_1f_4}\frac{f_1^2f_4^2}{f_2^4}\frac{f_6^{10}}{f_3^4f_{12}^4}
            \pmod{3} \\
            &\equiv \frac{f_1^6f_4^3}{f_2^9}\frac{f_6^{10}}{f_3^4f_{12}^4}
            \pmod{3} \\
            &\equiv \frac{f_3^2f_{12}}{f_6^3}\frac{f_6^{10}}{f_3^4f_{12}^4}
            \pmod{3} \\
            &\equiv \frac{f^7_6}{f^2_3f^3_{12}} \pmod{3}.
        \end{align*}
        Since this last expression is a function of $q^3$, we have our desired result.
    \end{proof}
\end{theorem}

We close this section by proving congruence families modulo 3 for infinitely many different functions $\OL{b}_k(n)$ along specific arithmetic progressions.  The proofs still follow from the elementary techniques demonstrated above.    

\begin{theorem}
    For all $l, m \geq 0$, $\OL{b}_{27l + 8}(27m + 18) \equiv 0 \pmod{3}$.
    % from 2025-11-17
    \begin{proof}
        We know
        \begin{align*}
        \sum_{n\geq0}\OL{b}_{27l+8}(n)q^{n} &= \frac{f_4^{27l+8}}{f_1^{54l+16}f_2^{27l+8}} \\
        &= \frac{f_4^{27l+9}}{f_1^{54l+15}f_2^{27l+9}}\frac{f_2}{f_1f_4}\\
        &\equiv \frac{f_{12}^{9l+3}}{f_3^{18l+5}f_6^{9l+3}}\frac{1}{\psi(-q)} \pmod{3}\\
        &\equiv \frac{f_{12}^{9l+3}}{f_3^{18l+5}f_6^{9l+3}}\left(\frac{\psi(-q^9)}{\psi(-q^3)^4}\left(P(-q^3)^2+qP(-q^3)\psi(-q^9)+q^2\psi(-q^9)^2\right)\right) \pmod{3}.
        \end{align*}
        Considering the terms of our $q$-series of the form $q^{3n}$, our congruence becomes
        $$\sum_{n\geq0}\OL{b}_{27l+8}(3n)q^{3n} \equiv \frac{f_{12}^{9l+3}}{f_3^{18l+5}f_6^{9l+3}}\frac{\psi(-q^9)}{\psi(-q^3)^4}P(-q^3)^2 \pmod{3}.$$
        Substituting $q$ for $q^{3}$, we now have
        \begin{align*}
            \sum_{n\geq0}\OL{b}_{27l+8}(3n)q^{n} &\equiv \frac{f_{4}^{9l+3}}{f_1^{18l+5}f_2^{9l+3}}\frac{1}{\psi(-q)}P(-q)^2 \pmod{3} \\
            &\equiv \frac{f_{12}^{3l+1}}{f_1^{18l+3}f_1^2f_6^{3l+1}}\frac{f_2}{f_1f_4}\left(\frac{f_1f_4f^5_6}{f^2_2f^2_3f^2_{12}}\right)^2 \pmod{3} \\
            &\equiv \frac{f_{12}^{3l+1}}{f_3^{6l+1}f_6^{3l+1}}\frac{f_2}{f_1^3f_4}\frac{f_1^2f_4^2}{f^4_2}\frac{f_6^{10}}{f_3^4f_{12}^4} \pmod{3} \\
            &\equiv \frac{f_{12}^{3l-3}}{f_3^{6l+5}f_6^{3l-9}}\frac{f_4}{f_1f_2^3} \pmod{3} \\
            &\equiv \frac{f_{12}^{3l-3}}{f_3^{6l+5}f_6^{3l-9}}\frac{1}{f_6}\frac{f_4}{f_1} \pmod{3} \\
            &\equiv \frac{f_{12}^{3l-3}}{f_3^{6l+5}f_6^{3l-8}}\frac{f_4}{f_1} \pmod{3} \\
            &\equiv \frac{f_{12}^{3l-3}}{f_3^{6l+5}f_6^{3l-8}}\sum_{n\geq 0}ped(n)q^n \pmod{3}
        \end{align*}
        using Theorem \ref{even parts distinct partition}.
        Using Theorem \ref{ped dissection} and considering the terms of our $q$-series of the form $q^{3n}$, our congruence becomes
        \begin{align*}
            \sum_{n\geq0}\OL{b}_{27l+8}(9n)q^{3n} &\equiv \frac{f_{12}^{3l-3}}{f_3^{6l+5}f_6^{3l-8}}\frac{f_{12}f_{18}^4}{f_3^3f_{36}^2} \pmod{3}\\
            &\equiv \frac{f_{12}^{3l-2}f_{18}^4}{f_3^{6l+8}f_6^{3l-8}f_{36}^2}\pmod{3}.
        \end{align*}
        Once more, substituting $q$ for $q^{3}$, we get
        \begin{align*}
            \sum_{n\geq0}\OL{b}_{27l+8}(9n)q^{n} &\equiv \frac{f_{4}^{3l-2}f_{6}^4}{f_1^{6l+8}f_2^{3l-8}f_{12}^2} \pmod{3} \\
            &\equiv \frac{f_{6}^4}{f_{12}^2} \frac{f_{4}^{3l-2}}{f_1^{6l+8}f_2^{3l-8}} \pmod{3} \\
            &\equiv \frac{f_{6}^4}{f_{12}^2} \frac{f_{4}^{3l-3}}{f_1^{6l+9}f_2^{3l-9}}\frac{f_1f_4}{f_2} \pmod{3} \\
            &\equiv \frac{f_{6}^4}{f_{12}^2} \frac{f_{12}^{l-1}}{f_3^{2l+3}f_6^{l-3}}\psi(-q) \pmod{3} \\
            &\equiv \frac{f_{6}^4}{f_{12}^2} \frac{f_{12}^{l-1}}{f_3^{2l+3}f_6^{l-3}}\left(P(-q^3)-q\psi(-q^9)\right) \pmod{3}.
        \end{align*}
Note that, in this last expression, there are no terms of the form $q^{3m+2}$.  Therefore, we know that, for all $m$, $\OL{b}_{27l + 8}(9(3m+2)) = \OL{b}_{27l + 8}(27m + 18) \equiv 0 \pmod{3}$.
    \end{proof}
\end{theorem}

\begin{theorem}
For all $l$ and $m \geq 0$, 
\begin{align*}
\OL{b}_{27l+10}(3m+2) &\equiv 0 \pmod{3}, \\
\OL{b}_{27l+10}(27m+18) &\equiv 0 \pmod{3}. 
\end{align*}
    \begin{proof}
        We know
        \begin{align*}
            \sum_{n\geq0}\OL{b}_{27l+10}(n)q^{n} &= \frac{f_4^{27l+10}}{f_1^{54l+20}f_2^{27l+10}} \\
            &= \frac{f_4^{27l+9}}{f_1^{54l+21}f_2^{27l+9}}\frac{f_1f_4}{f_2}\\
            &\equiv \frac{f_{12}^{9l+3}}{f_3^{18l+7}f_6^{9l+3}}\left(P(-q^3)-q\psi(-q^9)\right)\pmod{3}.
        \end{align*}
Note that the power series representation of the final expression above contains no terms of the form $q^{3n+2}$, so our first congruence holds. 

Next, considering the terms of our $q$-series of the form $q^{3n}$, our congruence becomes
        $$\sum_{n\geq0}\OL{b}_{27l+10}(3n)q^{3n} \equiv \frac{f_{12}^{9l+3}}{f_3^{18l+7}f_6^{9l+3}}P(-q^3) \pmod{3}.$$
        Substituting $q$ for $q^{3}$, we now have
        \begin{align*}
            \sum_{n\geq0}\OL{b}_{27l+10}(3n)q^{n}  &\equiv \frac{f_{4}^{9l+3}}{f_1^{18l+7}f_2^{9l+3}}P(-q) \pmod{3} \\
            &\equiv \frac{f_{4}^{9l+3}}{f_1^{18l+7}f_2^{9l+3}}\frac{f_1f_4f^5_6}{f^2_2f^2_3f^2_{12}} \pmod{3} \\
            &\equiv \frac{f_{4}^{9l+3}}{f_1^{18l+6}f_2^{9l+6}}\frac{f_2f_4f^5_6}{f^2_3f^2_{12}} \pmod{3} \\
            &\equiv \frac{f_{12}^{3l+1}}{f_3^{6l+2}f_6^{3l+2}}\frac{f^5_6}{f^2_3f^2_{12}}F(q^2) \pmod{3} \text{ \ \ from Theorem \ref{Capital F Dissection}} \\
            &\equiv \frac{f_{12}^{3l-1}}{f_3^{6l+4}f_6^{3l-3}}F(q^{18})\left(X(q^6)^{-1}-q^2-2q^4X(q^6)\right) \pmod{3}.
        \end{align*}
        Again considering the terms of our $q$-series of the form $q^{3n}$, our congruence becomes
        \begin{align*}
            \sum_{n\geq0}\OL{b}_{27l+10}(9n)q^{3n} &\equiv \frac{f_{12}^{3l-1}}{f_3^{6l+4}f_6^{3l-3}}F(q^{18})X(q^6)^{-1} \pmod{3}.
        \end{align*}
        Once more, substituting $q$ for $q^{3}$, we get
        \begin{align*}
            \sum_{n\geq0}\OL{b}_{27l+10}(9n)q^{n} &\equiv \frac{f_{4}^{3l-1}}{f_1^{6l+4}f_2^{3l-3}}F(q^{6})X(q^2)^{-1} \pmod{3} \\
            &\equiv \frac{f_{4}^{3l-1}}{f_1^{6l+4}f_2^{3l-3}}f_6f_{12}\frac{f_4f_6^3}{f_2f_{12}^3} \pmod{3} \\
            &\equiv \frac{f_{4}^{3l}}{f_1^{6l+3}f_2^{3l}}\frac{f_2^2}{f_1}\frac{f_6^4}{f_{12}^2} \pmod{3} \\
            &\equiv \frac{f_{12}^{l}}{f_3^{2l+1}f_6^{l}}\frac{f_6^4}{f_{12}^2}\left(P(q^3)+q\psi(q^9)\right) \pmod{3}.
        \end{align*}
        Finally, we consider the terms of our $q$-series of the form $q^{3m+2}$ to get our result.
    \end{proof}
\end{theorem}

\begin{theorem}
    For all $l$ and $m \geq 0$, $\OL{b}_{27l + 14}(27m + 18) \equiv 0 \pmod{3}$.
    % from 2025-11-19
    \begin{proof}
        We know
        \begin{align*}
            \sum_{n\geq0}\OL{b}_{27l+14}(n)q^{n} &= \frac{f_4^{27l+14}}{f_1^{54l+28}f_2^{27l+14}} \\
            &= \frac{f_4^{27l+15}}{f_1^{54l+27}f_2^{27l+15}}\frac{f_2}{f_1f_4}\\
            &\equiv \frac{f_{12}^{9l+5}}{f_3^{18l+9}f_6^{9l+5}}\left(\frac{\psi(-q^9)}{\psi(-q^3)^4}\left(P(-q^3)^2+qP(-q^3)\psi(-q^9)+q^2\psi(-q^9)^2\right)\right) \pmod{3}.
        \end{align*}
        Considering the terms of our $q$-series of the form $q^{3n}$, our congruence becomes
        \begin{align*}
             \sum_{n\geq0}\OL{b}_{27l+14}(3n)q^{3n} &\equiv \frac{f_{12}^{9l+5}}{f_3^{18l+9}f_6^{9l+5}}\frac{\psi(-q^3)^3}{\psi(-q^3)^3\psi(-q^3)}P(-q^3)^2 \pmod{3} \\
            &\equiv \frac{f_{12}^{9l+5}}{f_3^{18l+9}f_6^{9l+5}}\frac{1}{\psi(-q^3)}P(-q^3)^2 \pmod{3}.
        \end{align*}
        Substituting $q$ for $q^{3}$, we now have
        \begin{align*}
            \sum_{n\geq0}\OL{b}_{27l+14}(3n)q^{n} &\equiv \frac{f_{4}^{9l+5}}{f_1^{18l+9}f_2^{9l+5}}\frac{1}{\psi(-q)}P(-q)^2 \pmod{3} \\
            &\equiv \frac{f_{4}^{9l+5}}{f_1^{18l+9}f_2^{9l+5}}\frac{f_2}{f_1f_4}\left(\frac{f_1f_4f^5_6}{f^2_2f^2_3f^2_{12}}\right)^2 \pmod{3} \\
            &\equiv \frac{f_{4}^{9l+5}}{f_1^{18l+9}f_2^{9l+5}}\frac{f_2}{f_1f_4}\frac{f_1^2f_4^2f^{10}_6}{f^4_2f^4_3f^4_{12}} \pmod{3} \\
            &\equiv \frac{f_{4}^{9l+5}}{f_1^{18l+9}f_2^{9l+5}}\frac{f_1f_4}{f_2^3}\frac{f_6^{10}}{f_3^4f_{12}^4} \pmod{3} \\
            &\equiv \frac{f_{4}^{9l+6}}{f_1^{18l+9}f_2^{9l+6}}\frac{f_1}{f_2^2}\frac{f_6^{10}}{f_3^4f_{12}^4} \pmod{3} \\
            &\equiv \frac{f_{12}^{3l+2}}{f_3^{6l+3}f_6^{3l+2}}\frac{f_6^{10}}{f_3^4f_{12}^4}\frac{1}{\psi(q)} \pmod{3} \\
            &\equiv \frac{f_{12}^{3l-2}}{f_3^{6l+7}f_6^{3l-8}}\left(\frac{\psi(q^9)}{\psi(q^3)^4}\left(P(q^3)^2-qP(q^3)\psi(q^9)+q^2\psi(q^9)^2\right)\right) \pmod{3}.
        \end{align*}
        Again considering the terms of our $q$-series of the form $q^{3n}$, our congruence becomes
        \begin{align*}
            \sum_{n\geq0}\OL{b}_{27l+14}(9n)q^{3n} &\equiv \frac{f_{12}^{3l-2}}{f_3^{6l+7}f_6^{3l-8}}\frac{\psi(q^9)}{\psi(q^3)^4}P(q^3)^2 \pmod{3} \\
            &\equiv \frac{f_{12}^{3l-2}}{f_3^{6l+7}f_6^{3l-8}}\frac{\psi(q^3)^3}{\psi(q^3)^3\psi(q^3)}P(q^3)^2 \pmod{3} \\
            &\equiv \frac{f_{12}^{3l-2}}{f_3^{6l+7}f_6^{3l-8}}\frac{1}{\psi(q^3)}P(q^3)^2 \pmod{3}.
        \end{align*}
        Once more, substituting $q$ for $q^{3}$, we get
        \begin{align*}
            \sum_{n\geq0}\OL{b}_{27l+14}(9n)q^{n} &\equiv \frac{f_{4}^{3l-2}}{f_1^{6l+7}f_2^{3l-8}}\frac{1}{\psi(q)}P(q)^2 \pmod{3} \\
            &\equiv \frac{f_{4}^{3l-2}}{f_1^{6l+7}f_2^{3l-8}}\frac{f_1}{f_2^2}\left(\frac{f_2f_3^2}{f_1f_6}\right)^2 \pmod{3} \\
            &\equiv \frac{f_{4}^{3l-2}}{f_1^{6l+6}f_2^{3l-6}}\frac{f_2^2f_3^4}{f_1^2f_6^2} \pmod{3} \\
            &\equiv \frac{f_{4}^{3l-3}}{f_1^{6l+9}f_2^{3l-9}}\frac{f_1}{f_2f_4}\frac{f_3^4}{f_6^2}\pmod{3} \\
            &\equiv \frac{f_{12}^{l-1}}{f_3^{2l+3}f_2^{l-3}}\frac{f_3^4}{f_6^2}\left (P(-q^3)-q\psi(-q^9)\right) \pmod{3}.
        \end{align*}
        Finally, we consider the terms of our $q$-series of the form $q^{3m+2}$ to get our result.
    \end{proof}
\end{theorem}

\begin{theorem}
For all $l$ and $m\geq0$, $\OL{b}_{9l+2}(9m+3) \equiv 0 \pmod{3}$.
    \begin{proof}
        We know
        \begin{align*}
            \sum_{n\geq0}\OL{b}_{9l+2}(n)q^{n} &= \frac{f_4^{9l+2}}{f_1^{18l+4}f_2^{9l+2}} 
            = \frac{f_4^{9l+3}}{f_1^{18l+3}f_2^{9l+3}} \frac{f_2}{f_1f_4}
            \equiv \frac{f_{12}^{3l+1}}{f_3^{6l+1}f_6^{3l+1}} \frac{1}{\psi(-q)} \pmod{3}\\
            &\equiv \frac{f_{12}^{3l+1}}{f_3^{6l+1}f_6^{3l+1}} \frac{\psi(-q^9)}{\psi(-q^3)^4}\left(P(-q^3)^2+qP(-q^3)\psi(-q^9)+q^2\psi(-q^9)^2\right) \pmod{3}.
        \end{align*}
        Considering the terms of our $q$-series of the form $q^{3n}$, we have
        $$\sum_{n\geq0}\OL{b}_{9l+2}(3n)q^{3n} \equiv \frac{f_{12}^{3l+1}}{f_3^{6l+1}f_6^{3l+1}} \frac{\psi(-q^9)}{\psi(-q^3)^4}P(-q^3)^2 \pmod{3}.$$
        Substituting $q$ for $q^{3}$, we now have
        \begin{align*}
            \sum_{n\geq0}\OL{b}_{9l+2}(3n)q^{n} &\equiv \frac{f_{4}^{3l+1}}{f_1^{6l+1}f_2^{3l+1}} \frac{\psi(-q^3)}{\psi(-q)^4}P(-q)^2 \pmod{3} \\
            &\equiv \frac{f_{4}^{3l+1}}{f_1^{6l+1}f_2^{3l+1}} \frac{\psi(-q)^3}{\psi(-q)^4}\left(\frac{f_1f_4f^5_6}{f^2_2f^2_3f^2_{12}}\right)^2 \pmod{3} \\
            &\equiv \frac{f_{4}^{3l+1}}{f_1^{6l+1}f_2^{3l+1}} \frac{1}{\psi(-q)}\frac{f_1^2f_4^2f^{10}_6}{f^4_2f^4_3f^4_{12}} \pmod{3} \\
            &\equiv \frac{f_{4}^{3l+1}}{f_1^{6l+1}f_2^{3l+1}} \frac{f_2}{f_1f_4}\frac{f_1^2f_4^2}{f^4_2}\frac{f^{10}_6}{f^4_3f^4_{12}}  \pmod{3} \\
            &\equiv \frac{f_{4}^{3l+2}}{f_1^{6l}f_2^{3l+4}} \frac{f^{10}_6}{f^4_3f^4_{12}}  \pmod{3} \\
            &\equiv \frac{f_{4}^{3l+3}}{f_1^{6l}f_2^{3l+3}} \frac{1}{f_2f_4} \frac{f^{10}_6}{f^4_3f^4_{12}}  \pmod{3} \\
            &\equiv \frac{f_{12}^{l+1}}{f_3^{2l}f_6^{l+1}} \frac{f^{10}_6}{f^4_3f^4_{12}}\frac{1}{F(q^2)}  \pmod{3} \\
            &\equiv \frac{f^{l-3}_{12}}{f^{2l+4}_3f^{l-9}_6} \frac{1}{F(q^2)}\pmod{3}.
        \end{align*}
        Using Theorem \ref{cubic 3-dissection} and substituting $q^2$ for $q$, we get
        \begin{align*}
            \frac{1}{F(q^2)} &= \frac{F(q^{18})^3}{F(q^6)^4}\left(X(q^6)^{-2}+q^2X(q^6)^{-1} +3q^4-2q^6X(q^6)+4q^8X(q^6)^2\right) \\
            &= \frac{F(q^{18})^3}{F(q^6)^4}\left(\left(X(q^6)^{-2}-2q^6X(q^6)\right)+q\left(3q^3\right)+q^2\left(X(q^6)^{-1}+4q^6X(q^6)^2\right)\right).
        \end{align*}
        Considering the terms of our $q$-series of the form $q^{3m+1}$, our congruence above becomes
        \begin{align*}
            \sum_{m\geq0}\OL{b}_{9l+2}(9m+3)q^{3m+1} &\equiv 3q^4\frac{f^{l-3}_{12}}{f^{2l+4}_3f^{l-9}_6} \frac{F(q^{18})^3}{F(q^6)^4} \pmod{3}\\
            &\equiv 0 \pmod{3}.
        \end{align*}
    \end{proof}
\end{theorem}

\begin{theorem}
For all $l$ and $m\geq0$, 
\begin{align*}
\OL{b}_{9l+7}(3m+2) &\equiv 0 \pmod{3}, \\
\OL{b}_{9l+7}(9m+6) &\equiv 0 \pmod{3}. 
\end{align*}
    \begin{proof}
        We know
        \begin{align*}
            \sum_{n\geq0}\OL{b}_{9l+7}(n)q^{n} &= \frac{f^{9l+7}_4}{f^{2(9l+7)}_1f^{9l+7}_2} 
            = \frac{f^{9l+6}_4}{f^{18l+15}_1f^{9l+6}_2} \frac{f_1f_4}{f_2} 
            \equiv \frac{f^{3l+2}_{12}}{f^{6l+5}_3f^{3l+2}_6} \psi(-q) \pmod{3}\\
            &\equiv \frac{f^{3l+2}_{12}}{f^{6l+5}_3f^{3l+2}_6} \left(P(-q^3) - q\psi(-q^9) \right) \pmod{3}
        \end{align*}
using Equation \eqref{3-Dissection Psi}.
Note that the power series representation of the final expression above contains no terms of the form $q^{3n+2}$, so our first congruence holds. 

        Considering the terms of our $q$-series of the form $q^{3n}$, our congruence becomes
        \begin{align*}
            \sum_{n\geq0}\OL{b}_{9l+7}(3n)q^{n} &\equiv \frac{f^{3l+2}_{12}}{f^{6l+5}_3f^{3l+2}_6} P(-q^3) \pmod{3} \\
            &\equiv \frac{f^{3l+2}_{4}}{f^{6l+5}_1f^{3l+2}_2} P(-q) \pmod{3} \\
            &\equiv \frac{f^{3l+2}_{4}}{f^{6l+5}_1f^{3l+2}_2} \frac{f_1f_4f^5_6}{f^2_2f^2_3f^2_{12}}  \pmod{3} \\
            &\equiv \frac{f^{3l+3}_{4}}{f^{6l+4}_1f^{3l+4}_2} \frac{f^{15}_2}{f^6_1f^6_{4}} \pmod{3} \\
            &\equiv \frac{f^{3l-3}_{4}}{f^{6l+10}_1f^{3l-11}_2}  \pmod{3} \\
            &\equiv \frac{f^{3l-3}_{4}}{f^{6l+9}_1f^{3l-9}_2}\frac{f^{2}_2}{f_1}  \pmod{3} \\
            &\equiv \frac{f^{3l-3}_{4}}{f^{6l+9}_1f^{3l-9}_2}\psi(q)  \pmod{3} \\
            &\equiv \frac{f^{l-1}_{12}}{f^{2l+3}_3f^{l-3}_6}\left(P(q^3) + q\psi(q^9) \right)  \pmod{3}.
        \end{align*}
        Finally, we consider the terms of our $q$-series of the form $q^{3m+2}$ to get our result.
    \end{proof}
\end{theorem}

\begin{theorem}
    For all $l$ and $m\geq0$, $\OL{b}_{3l+1}(3m+2) \equiv 0 \pmod{3}$.
    \begin{proof}
        We know
        \begin{align*}
            \sum_{n\geq0}\OL{b}_{3l+1}(n)q^{n} &= \frac{f^{3l+1}_4}{f^{2(3l+1)}_1f^{3l+1}_2} 
            = \frac{f^{3l}_4}{f^{6l+3}_1f^{3l}_2} \frac{f_1f_4}{f_2} 
            \equiv \frac{f^{l}_{12}}{f^{2l+1}_3f^{l}_6} \psi(-q) \pmod{3}\\
            &\equiv \frac{f^{l}_{12}}{f^{2l+1}_3f^{l}_6} \left(P(-q^3) - q\psi(-q^9) \right) \pmod{3}.
        \end{align*}
        Finally, we consider the terms of our $q$-series of the form $q^{3m+2}$ to get our result.
    \end{proof}
\end{theorem}

\section{Concluding Thoughts}
\label{sec:concluding_thoughts}
%We close by highlighting two sets of thoughts.  First, it is very easy to prove the following infinite family of congruences modulo $p$ for any prime $p$.  
We close by highlighting the following infinite family of congruences modulo $p$ for any prime $p$, the proof of which is extremely straightforward.  

\begin{theorem}
    For $p$ prime, $l \geq 1$, $m \geq 0$, and $a $ such that $1 \leq a < p$, $\OL{b}_{pl}(pm+a) \equiv 0 \pmod{p}$.
    \begin{proof}
        Let $p$ be prime and $l \geq 1$.
    Using Lemma \ref{Expanded Freshman's Gn Fn}, note that 
        $$\sum_{n\geq 0} \OL{b}_{pl}(n)q^n = \frac{f_4^{pl}}{f_1^{2pl}f_2^{pl}} \equiv  \frac{f_{4p}^l}{f_p^{2l}f_{2p}^l} \pmod{p}.$$
        Our result immediately follows since
        $\frac{f_{4p}^l}{f_p^{2l}f_{2p}^l}$
        is a function of $q^p$.
  \end{proof}
\end{theorem}

%Secondly, in \cite[Theorem 3.1]{PujashreeBuragohain}, Buragohain and Saikia note that 
%$$\OL{b}_{2^\alpha l + j}(n)\equiv \OL{b}_{j}(n) \pmod{2^{\alpha+1}}$$
%for $l,n \geq 0$ and $\alpha, j \geq 1$. This result also holds for $\alpha = 0$ and $j = 0$, and this leads to a significant family of divisibility properties modulo arbitrarily large powers of 2.  

%\begin{theorem}
%For $\alpha, l \geq 0$ and $n \geq 1$, 
%$$\OL{b}_{2^\alpha l }(n)\equiv 0 \pmod{2^{\alpha+1}}.$$
%\end{theorem}

%Some of the congruences modulo small powers of 2 which we mentioned above (for example, some of the work of Shivaprasada Nayaka et al. \cite{TheOG}) appear to follow from the above theorem.  

\section*{Declarations} 
\noindent 
The authors make the following declarations.
\begin{enumerate}
    \item The authors declare no conflicts of interest.
    \item The authors have no relevant financial or non-financial interests to disclose.
    \item Both authors contributed substantially to this project, including reading and approving the final manuscript.
\end{enumerate}

%\newpage

\end{document}